\documentclass[a4paper,11pt]{amsart}
\usepackage{natbib}
\usepackage{hyperref}
\usepackage{mathscinet}
\usepackage{mathrsfs}

\newtheorem{theorem}{Theorem}
\theoremstyle{plain}
\newtheorem{lem}{Lemma}
\newtheorem{proposition}{Proposition}
\theoremstyle{remark}
\newtheorem{rmk}{Remark}
\numberwithin{equation}{section}
\DeclareMathOperator{\Var}{\mathbb{V}ar}
\DeclareMathOperator{\Corr}{\mathbb{C}orr}
\DeclareMathOperator{\Vol}{Vol}
\DeclareMathOperator{\diag}{diag}
\DeclareMathOperator{\diam}{diam}
\newcommand{\ee}{\mathbb E}
\newcommand{\nn}{\mathbb N}
\newcommand{\pp}{\mathbb P}
\newcommand{\rr}{\mathbb R}
\newcommand{\cn}{\mathcal{N}}
\newcommand{\half}{\frac{1}{2}}
\newcommand{\toi}{\to\infty}
\newcommand{\as}[1]{\quad\text{as}\quad #1\toi}

\newcommand{\pto}{\stackrel{\pp}{\to}}
\newcommand{\scal}[2]{\left\langle #1, #2 \right\rangle}
\newcommand{\bs}{\boldsymbol}
\newcommand{\bt}{{\boldsymbol t}}
\begin{document}
\bibliographystyle{plainnat}
\setcitestyle{numbers}
\title[Extremes of multidimensional Gaussian processes]
{Extremes of multidimensional \\Gaussian processes}

\author{K.\ D\polhk{e}bicki}
\address{Mathematical Institute, University of
Wroc\l aw, pl.\ Grunwaldzki 2/4, 50-384 Wroc\l aw, Poland.}
\email{Krzysztof.Debicki@math.uni.wroc.pl}

\author{K.M.\ Kosi\'nski}
\address{E{\sc urandom}, Eindhoven University of Technology, the Netherlands;
Korteweg-de Vries Institute for Mathematics, University of Amsterdam, the Netherlands.}
\email{K.M.Kosinski@uva.nl}
\thanks{The first and third authors thank the Isaac Newton Institute, Cambridge, UK, for hospitality. 
The research of the first and the fourth authors was supported by MNiSW Research Grant N N201 394137 (2009-2011).
The second author thanks the Mathematical Institute, University of Wroc\l aw, Poland, for hospitality.
The research of the second author was supported by NWO grant 613.000.701.}

\author{M.\ Mandjes}
\address{Korteweg-de Vries Institute for Mathematics,
University of Amsterdam, the Netherlands; E{\sc urandom},
Eindhoven University of Technology, the Netherlands; CWI,
Amsterdam, the Netherlands}
\email{M.R.H.Mandjes@uva.nl}

\author{T.\ Rolski}
\address{Mathematical Institute, University of
Wroc\l aw, pl.\ Grunwaldzki 2/4, 50-384 Wroc\l aw, Poland.}
\email{Tomasz.Rolski@math.uni.wroc.pl}

\date{September 7, 2010}

\begin{abstract}
This paper considers extreme values attained by a centered, multidimensional Gaussian process
$X(t)= (X_1(t),\ldots,X_n(t))$ minus drift $d(t)=(d_1(t),\ldots,d_n(t))$, on an arbitrary set $T$.
Under mild regularity conditions, we establish the asymptotics of
\[
\log\pp\left(\exists{t\in T}:\bigcap_{i=1}^n\left\{X_i(t)-d_i(t)>q_iu\right\}\right),
\]
for positive thresholds $q_i>0$, $i=1,\ldots,n$ and $u\toi$.
Our findings generalize and extend previously known results for the single-dimensional and two-dimensional cases.
A number of examples illustrate the theory.
\end{abstract}

\maketitle

\section{Introduction}
\label{Intro}
\noindent
Owing to its relevance in various application domains, in the theory of stochastic processes,
substantial attention has been paid to estimating the tail distribution of the maximum value attained.
 In mathematical terms, the setting considered involves an $\rr$-valued stochastic process
$X=\{X(t): t\in T\}$ for some arbitrary set $T$ and a threshold level $u> 0$, where the focus is on characterizing the probability
\begin{equation}
\label{eq:pu}
\pp\left(\sup_{t\in T} X(t)>u\right)=\pp\left(\exists{t\in T}: X(t)>u\right).
\end{equation}
More specifically, the case in which $X$ is a Gaussian process has been studied in detail.
This hardly led to any explicit results for \eqref{eq:pu},
but there is quite a large body of literature on results  for the asymptotic regime in which  $u$ grows large.
The prototype case dealt with a centered Gaussian
process with bounded trajectories for which the {\it logarithmic asymptotics} were found: it was shown that
\begin{equation}
\label{eq:log}
\lim_{u\toi}u^{-2}\log \pp\left(\sup_{t\in T}X(t)>u\right)=-\left(2\sigma_T^2\right)^{-1},
\end{equation}
where
\[
\sigma_T^2:=\sup_{t\in T} \ee X^2(t).
\]
See \citet[p.\ 42]{Adler90} or \citet[Section 12]{Lifshits95} for
this and related results. The monographs \citet{Lifshits95} and
\citet{Piterbarg96} contain more refined results: under appropriate conditions, an explicit  function $\phi(u)$
is given such that the ratio of \eqref{eq:pu} and $\phi(u)$ tends to 1 as $u\toi$
(so-called {\it exact asymptotics}).
The logarithmic asymptotics \eqref{eq:log} can easily be extended to
the case of noncentered Gausssian processes if the mean function is bounded.
The situation gets interesting if both trajectories and the mean function of the process
are unbounded.
In this respect we mention \citet{Duffield95} and \citet{Debicki99}, where the logarithmic asymptotics of $\pp(\sup_{t\ge0} (X(t)-d(t))>u)$ for
general centered Gaussian processes $X$, under some regularity assumptions on the drift function $d$,
were derived; see also \citet{Husler99}, \citet{Dieker05a} and references therein.

While the above results all relate to one-dimensional suprema, considerably less attention has been paid to their multidimensional counterparts.
One of few exceptions is provided by the work of
\citet{Piterbarg05}, who considered the case of two $\rr$-valued, possibly dependent, centered Gaussian processes
$\{X_1(t_1):t_1\in T_1\}$ and $\{X_2(t_2):t_2\in T_2\}$. They found the
logarithmic asymptotics of
\begin{equation}
\label{eq:PS}
\pp(\exists{(t_1,t_2)\in T}: X_1(t_1)>u, X_2(t_2)>u)
\end{equation}
for some $T\subseteq T_1\times T_2$, under the assumption that the trajectories of $X_1$ and $X_2$ are bounded.

In this paper our objective is to obtain the logarithmic asymptotics of (following the convention that vectors are written in bold)
\begin{equation}
\label{goal}
P(u):=\pp \left( \exists{\bs t\in T}:\bigcap_{i=1}^n \{X_i(\bs t)-d_i(\bs t)>q_i u\}\right);
\end{equation}
here
$\{\bs X(\bt):\bt\in T\}$, with $\bs X(\bt)= (X_1(\bt),\ldots,X_n(\bt))'$, is an $\rr^n$-valued centered Gaussian processes defined
on an arbitrary set $T\subseteq \rr^m$, for some $m,n\in\nn$, the $d_i(\cdot)$ are drift functions
and $q_i>0$ are threshold levels, $i=1,\ldots,n$.
Our setup is rich enough to cover both of the cases in which $P(u)$ corresponds to the event
in which (i)~it is required that there is a {\it single} time epoch $t\in\rr$
such that $X_i(t)-d_i(t)>q_i u$ for all $i=1,\ldots,n$ and (ii)~there are $n$ epochs $(t_1,\ldots,t_n)$ such that
$X_i(t_i)-d_i(t_i)>q_i u$ for  all $i=1,\ldots,n$. We get back to this issue in detail in Remark \ref{rem:sim},
where it is also noted that the theory covers a variety of situations between these two extreme situations.

Compared to the one-dimensional setting, the multidimensional case requires various technical complications to be settled.
The derivations of logarithmic asymptotics usually rely on an upper and lower bound,
where the latter is based on the inequality
\[
P(u)\ge
\sup_{{\bs t}\in T}\pp \left(\bigcap_{i=1}^n \{X_i({\bs t})-d_i({\bs t})>q_i u\}\right).
\]
Strikingly, in terms of the logarithmic asymptotics,
this lower bound is actually tight, which is essentially due to the common `large deviations heuristic':
the decay rate of the probability of a union of events coincides with the decay rate of the most likely event among these events.
A first contribution of the present paper is that we show that this argument essentially carries over to the multidimensional setting.
In order to obtain the lower bound one needs asymptotics
of tail probabilities that correspond to multivariate normal distributions. 
In this domain a wealth of results are available (see, e.g., \citet{Hashorva05}
and references therein), but for our purposes we need estimates which are, in some specific sense, uniform. 
A version of such estimates, that is
tailored to our needs, is presented in Lemma \ref{lemma:lower}.

The upper bound is based on what we call a `saddle point equality' presented in Lemma \ref{lemma:saddle}.
It essentially allows us to approximate suprema of multidimensional
Gaussian process $\bs X$ by a specific one-dimensional Gaussian process,
namely a properly weighted sum of the coordinates $X_i$
of $\bs X$. Formally, we identify weights $w_i=w_i(t,u)\ge 0$ such that
the inequality
\[
P(u)\le
\pp\left(\exists{{\bs t}\in T}:\sum_{i=1}^n w_iX_i(\bt)>\sum_{i=1}^n w_i(uq_i+d_i(\bt))\right),
\]
is logarithmically asymptotically exact, as $u\to\infty$.
The reduction of the dimension of the problem allows us to use
one-dimensional techniques
(such as the celebrated Borell inequality).
Interestingly, the optimal weights can be interpreted in terms of
the solution to a convex programming problem that corresponds to an associated Legendre transform of
the covariance matrix of $\bs X$.
A different weighting technique has been developed in
\citet{Piterbarg05} for the case $n=2$, but without a motivation for the weights chosen.
We recover the result from \cite{Piterbarg05} in Remark \ref{rmk:Pit}.
Our analysis of \eqref{goal} extends the results from \cite{Debicki99,Piterbarg05},
in the first place because $\rr^n$-valued Gaussian processes are covered (for arbitrary $n\in\nn$).
The other main improvement relates to the considerable
generality in terms of the drift functions allowed; these were not covered in \cite{Piterbarg05}.

The paper is organized as follows. In Section \ref{MaP} we introduce  notation, describe in detail
objects of main interest to us, and state our main result;
we also pay special attention to the rationale behind the assumptions that we impose. In Section \ref{Ex}
we illustrate the main theorem by presenting a number of examples; one of these relates to
Gaussian processes with regularly varying variance functions. We also explain the potential
application of our result in queueing and insurance theory. In Section \ref{PMT} we
describe how the multidimensional process $\bs X$ can be approximated by a one-dimensional
process $Z$, obtained by appropriately weighting  the coordinates $X_i$. We prove some preliminary results about the
characteristics of the process $Z$. This section also contains the
saddle point equality mentioned above, Lemma \ref{lemma:saddle}, which is the crucial element of the proof of our main result.
Section \ref{PMT}  also contains all other  lemmas needed
to prove Theorem \ref{thm:main}, as well as the proof of our main result itself.

\section{Model, notation, and the main theorem}\label{MaP}
In this section we formally introduce the model, state the main theorem, and provide the
intuition behind the assumptions imposed.

\subsection{Model and notation}
\noindent Let $T\subseteq\rr^m$, for some $m\in\nn$. In this paper we consider
an $\rr^n$-valued (separable) centered Gaussian process $\bs X\equiv\{\bs
X(\bt),\bt\in T\}$ given by $\bs
X(\bt)=(X_1(\bt),\ldots, X_n(\bt))'$.
Let the so-called drift function be denoted by
$\bs d(\bt)=(d_1(\bt),\ldots,d_n(\bt))'$. Now, denote the covariance
matrix of $\bs X(\bt)$ by $\Sigma_{\bt}$. Throughout the paper
it is assumed that the matrix $\Sigma_{\bt}$ is invertible for every $\bt\in T$.
Here and in the sequel, we use the following notation and conventions:
\begin{itemize}
\item[$\cdot$]
We say $\bs v\ge  \bs w$ if $v_i\ge w_i$ for all $i=1,\ldots, n.$
\item[$\cdot$]
We write $\diag(\bs v)$ for the diagonal matrix with $v_i$ on the $i$th position of the diagonal.
\item[$\cdot$]
We define
$\bs v\bs w:=\diag(\bs v) \bs w' =(v_1w_1,\ldots, v_n w_n)'.$
\item[$\cdot$]
For $a\in\rr$, we let $\bs i(a)$ be an $n$-dimensional vector $(a,\ldots,a)'$ and also let
$\bs 0 =(0,\ldots,0)'$.
\item[$\cdot$]
We adopt the usual definitions of norms of vectors
$\|\bs x \|:=(\scal {\bs x} {\bs x})^{1/2}$,
where $\scal\cdot\cdot$ is the Euclidean inner product.
\item[$\cdot$]
We let $f(u)\sim g(u)$ denote that $\lim_{u\toi}f(u)/g(u)=1.$
\item[$\cdot$] We write
$\rr^n_+:=\{\bs x\in\rr^n:\bs x\ge0, \bs x\ne\bs 0\}.$
\end{itemize}
Throughout the paper not all vectors are of dimension $n$ (for instance $\bs t$ is
of dimension $m$), but the above notation should be understood with obvious changes.

With each $\Sigma_{\bt}$ we associate the matrix
$K_{\bt}=(k_{i,j}(\bt))_{i,j\le n}$, defined as
\[
K_{\bt} =\diag(\partial_{1,1}^{-1/2}(\bt),\ldots, \partial_{n,n}^{-1/2}(\bt)) \Sigma_{\bt}^{-1}
\diag(\partial_{1,1}^{-1/2}(\bt),\ldots, \partial_{n,n}^{-1/2}(\bt))
\]
with $\Sigma_{\bt}^{-1}=(\partial_{i,j}(\bt))_{i,j\le n}$.
We mention that $k_{i,j}(\bt)\in[-1,1]$ and that
$-k_{i,j}(\bt)$ is commonly interpreted as some sort of partial correlation between $X_i(\bt)$ and $X_j(\bt)$
controlling all other variables $X_k(\bt)$, $k\neq i,j$.

\subsection{Main result}
Throughout the paper, we impose the following assumptions.

\vspace{2mm}

\noindent {\rm \bf A1}\ \ \ $\sup_{\bt\in T} k_{i,j}(\bt) <1$ for all $i\ne j,$ $i,j=1,\ldots,n.$

\vspace{2mm}

\noindent
{\rm \bf A2}\ \ \ $\sup_{\bt\in T}(X_i(\bt)-\varepsilon d_i(\bt))<\infty$ {a.s. for all} $i=1,\ldots,n$ and all
$\varepsilon\in(0,1]$.

\vspace{2mm}

If a process $\bs X$ and a drift function $\bs d$ comply with assumptions 
{\bf A1}-{\bf A2}, then to shorten the notation, we will write
that $(\bs X,\bs d)$ satisfies {\bf A1}-{\bf A2}.

For a point $\bt\in T$ and a vector $\bs q>\bs 0$, define
\begin{align*}
M_{\bs X,\bs d,\bs q}(u,\bt)&:=\inf_{ \bs v\ge u \bs q} \scal{ \bs v+ \bs d(\bt)}{\Sigma_{\bt}^{-1} ( \bs v+ \bs d(\bt))},\\
M_{\bs X,\bs d,\bs q}(u;T)&:=\half\inf_{\bt\in T}M_{\bs X,\bs d,\bs q}(u,\bt).
\end{align*}

With these preliminaries we are ready to state our main result.
The following theorem can be seen as an $n$-dimensional extension of
\cite[Theorem 1]{Piterbarg05} and  \cite[Theorem 2.1]{Debicki99}.

\begin{theorem}
\label{thm:main}
Assume that $(\bs X,\bs d)$ satisfies {\bf A1}-{\bf A2}. Then, for any $\bs q >\bs 0$,
\begin{equation}
\label{thm:main:eq:2}
\log\pp\left(\exists{\bt\in T}:\bs X(\bt)-\bs d(\bt)>u \bs q\right)
\sim -M_{\bs X,\bs d,\bs q}(u;T)\as u.
\end{equation}
\end{theorem}

\begin{rmk}
\label{rem:sim}
The result stated in Theorem \ref{thm:main} enables us to
analyze, with $T_i\subseteq\rr$,
\begin{eqnarray}
\pp \left(\bigcap_{i=1}^n\left\{\sup_{t_i\in T_i}\left( X_i(t_i)-d_i(t_i) \right)>u q_i\right\}\right).\label{multi_sup}
\end{eqnarray}
To see this, let
$T:=T_1\times\ldots\times T_n$.
Also define processes
$\{Y_i(\bs t):\bs t\in T\}$, $i=1,\ldots, n$, such that
$Y_i(\bs t):= X_i(t_i)$, for $i=1,\ldots,n$.
Analogously, let $m_i(\bs t):= d_i(t_i)$, $i=1,\ldots,n$.
Then \eqref{multi_sup} equals
\[
\pp\left( \exists{\bs t\in T}:\bs  Y (\bs t)-\bs m(\bs t)>u \bs q \right),
\]
which, under the proviso that {\bf A1}- {\bf A2} are complied with by the newly constructed
$(\bs Y,\bs m)$, fits in the framework of Theorem \ref{thm:main}. This example naturally 
extends to the situation where the sets $T_i$ are of dimension higher than 1.
\end{rmk}

\subsection{Discussion of the assumptions}

In this subsection we motivate the assumptions that we imposed.

\begin{rmk}
\label{rmk:mainassum}Assumption {\bf A1} plays a crucial role in the proof of Lemma \ref{lemma:lower}. It can be geometrically interpreted  as follows. For a fixed $\bt\in T$, the distribution of $\bs X(\bt)$
equals that of $B_{\bt}\, \cn$, where $B_{\bt}$ is a matrix such that $\Sigma_{\bt}= B_{\bt}B_{\bt}'$
and $\cn$ is an $\rr^n$-valued standard normal random variable.
For some quadrant $Q_{\bt}$, we need  in the proof of Lemma~\ref{lemma:lower}
a lower estimate of $\pp(\bs X(\bt) \in Q_{\bt})=\pp(\cn\in B_{\bt}^{-1} Q_{\bt})$.
For $i=1,\ldots,n$ let $\bs e_i$ be, as usual, the standard basis vectors of $\rr^n$.
Then the cosine of the angle $\alpha_{i,j}$ between $B_{t}^{-1} \bs e_i$ and $B_{t}^{-1} \bs e_j$ is given by
\[
\cos(\alpha_{i,j})=\frac{\scal {B_{\bt}^{-1} \bs e_i}{B_{\bt}^{-1} \bs e_j}}{\|B_{\bt}^{-1} \bs e_i\|\|B_{\bt}^{-1} \bs e_j\|}
=\frac{\scal {\bs e_i} {\Sigma_{\bt}^{-1} \bs e_j}}{\|B_{\bt}^{-1} \bs e_i\|\|B_{\bt}^{-1} \bs e_j\|}=
\frac{\partial_{i,j}(\bt)}{\sqrt{\partial_{i,i}(\bt)\partial_{j,j}(\bt)}}=k_{i,j}(\bt).
\]
We thus observe that {\bf A1} entails that, for all $\bt\in T$, there is no pair of vector $B_{\bt}^{-1} \bs e_i$ and $B_{\bt}^{-1} \bs e_i$, with $i\ne j$, that
`essentially coincide', i.e., the angles remain bounded away from 0. Therefore, for any $\bs x \in B_{\bt}^{-1} Q_{\bt}$,
one can always find a set $A_{\bt}$ such that $\bs x\in A_{\bt}\subset B_{\bt}^{-1} Q_{\bt}$ and
$A_{\bt}$ has a diameter that is bounded, and a volume that is bounded away from zero, {\it uniformly} in $\bt \in T$.
\end{rmk}

\begin{rmk}
For $\varepsilon=1$, assumption {\bf A2} assures that
the event
\[
\bigcup_{\bt\in T}\{\bs X(\bt)-\bs d(\bt)>u\bs q\}\]
is not satisfied  trivially.
The following example shows that if {\bf A2} is not complied with, then it is not ensured that we remain
in the realm of exponential decay.
Consider a one-dimensional case in which $X\equiv\{X(t):t\ge0\}$ is a standard Brownian motion, and
for any $\delta>0$ let $d(t):=(1+\delta)\sqrt{2t\log\log t}$.
From the law of the iterated logarithm we conclude that the process
$X$ does not satisfy {\bf A2} for every $\varepsilon\in(0,1]$.
On the other hand we have (take $t:=u^4$)
\[
\pp\left(\sup_{t\ge0} \left(X(t)-(1+\delta)\sqrt{2t\log\log t}\right)>u\right)
\ge
\pp\left(\frac{u\cn}{1+(1+\delta)u\sqrt{2\log(4\log u)}}>1\right),
\]
where here $\cn$ is the real-valued standard normal random variable. On the logarithmic
scale the latter probability behaves roughly, for $u$ large, as
\[
-(1+\delta)^2\log\log u.
\]
For the case of $n=1$, {\bf A2} has been required in \cite[Theorem 2.1]{Debicki99} as well.
\end{rmk}

\begin{rmk}
\label{r4}
The drift functions $d_i$, $i=1,\ldots,n$, are not assumed to be increasing, but under assumption {\bf A2} we have
$\ell_i:=\inf_{\bt\in T} d_i(\bt)>-\infty$.
Because we are interested
in the asymptotic behavior of the probability in \eqref{thm:main:eq:2} as $u\toi$,
we can assume that $u>u_0:=-\min_i(\ell_i/q_i)$, and therefore the coordinates of $u \bs q+\bs d(\bt)$
stay positive for all $\bt\in T$. In what follows we shall always assume that $u>u_0$.
\end{rmk}

\section{Examples}
\label{Ex}
\noindent
In this section we present examples that demonstrate the consequences
of Theorem \ref{thm:main}. We focus on computing the decay rate
$M_{\bs X,\bs d,\bs q}(u;T)$ in two cases: (i)~the case of $\bs X$ having bounded sample paths a.s.;
(ii)~the case of the $X_i$ having stationary increments, regularly varying
variance functions, and $d_i(\cdot)$ being linear.
While in the former example the drift functions do not influence the asymptotics,
in the latter example the drifts {\it do have an} impact on the decay rate.

\subsection{Bounded sample paths and drift function}
We here analyze the case of $(\bs X,\bs d)$ satisfying

\vspace{2mm}

{\bf B1}\ \ \
The process $\bs X$ has bounded sample paths a.s.

\vspace{2mm}

{\bf B2}\ \ \
There exists $D<\infty$ such that $|d_i(\bt)|\le D$ for all $\bt\in T$ and $i=1,\ldots,n$.

\vspace{2mm}

We note that under {\bf B1}-{\bf B2}, it trivially holds that
assumption {\bf A2} is complied with as well. Assumptions {\bf B1}-{\bf B2}
are satisfied when $T$ is compact,
$\bs X$ has continuous sample paths a.s. and $\bs d$ is continuous for
instance. Let us introduce the following notation
\[
I_{\bs X,\bs q}(T):=\inf_{\bt\in T}\inf_{\bs v\ge \bs q}\scal {\bs v} {\Sigma_{\bt}^{-1} \bs v}.
\]
The following corollary is an immediate consequence of Theorem \ref{thm:main}.
\begin{proposition}\label{cor:bounded}
Assume that $(\bs X,\bs d)$ satisfies
{\bf A1} and {\bf B1}-{\bf B2}. Then,
\[
\log\pp\left(\exists{\bt\in T}:\bs X(\bt)-\bs d(\bt)>u \bs q \right)
\sim -\frac{u^2}{2}
I_{\bs X,\bs q}(T),\as u.
\]
\end{proposition}
The above proposition states that in the `bounded case' that we are currently considering, we encounter the same asymptotic decay
as in the driftless case ($\bs d\equiv \bs 0$).

\begin{rmk}
\label{rmk:Pit}
Some special cases of Proposition \ref{cor:bounded} have been treated before in the literature.
In particular,
let $X_1\equiv\{X_1(t_1):t_1\in T_1\}$ and $X_2\equiv\{X_2(t_2):t_2\in T_2\}$ be two centered
and bounded $\rr$-valued Gaussian processes. We introduce the notation  $\sigma_i(t_i):=
\sqrt{\Var(X_i(t_i))}$, $r(\bt):=\Corr(X_1(t_1),X_2(t_2))$ and also
\[
c_{\bs q}(\bt):=\min\left\{\frac{q_1}{\sigma_1(t_1)}
\frac{\sigma_2(t_2)}{q_2},\frac{\sigma_1(t_1)}{q_1}\frac{q_2}{\sigma_2(t_2)}\right\}.
\]
Then, upon combining Proposition \ref{cor:bounded} with Remark \ref{rem:sim}, we obtain, with $T\subseteq T_1\times T_2$,
\begin{eqnarray*}
\lefteqn{
\log\pp\left(\exists{(t_1,t_2)\in T}: X_1(t_1)>q_1u,X_2(t_2)>q_2u\right)
}\\
&\sim&
-\frac{u^2}{2}
\inf_{(t_1,t_2)\in T} \frac{1}{\left(\min\left\{\sigma_1(t_1)/q_1,\sigma_2(t_2)/q_2\right\}\right)^2}
\left(1+\frac{(c_{\bs q}(\bt)-r(\bt))^2}{1-r^2(\bt)}1_{\{r(\bt)<c_{\bs q}(\bt)\}}\right),
\end{eqnarray*}
as $u\to\infty$.
Observe that the above formula is also valid for $r(\bt)=\pm1$.
This recovers the result of \citet{Piterbarg05}.
\end{rmk}

\subsection{Stationary increments, linear drift}
\label{EX:rv}
This section focuses on the logarithmic asymptotics of
$\{\bs X(t)-\bs i (t) :t\ge 0\}$, where
$\bs X(t)=S\bs Y(t)$ for some invertible matrix $S$
and, as usual, $\bs Y(t)=(Y_1(t),\ldots,Y_n(t))'$. We assume that,
for $i=1,\ldots, n$,

\vspace{2mm}

{\bf C1} \ \ \ $\{Y_i(t):t\ge0\}$ are mutually independent, $\rr$-valued, centered Gaussian processes with stationary increments.

\vspace{2mm}

{\bf C2} \ \ \ The variance functions $\sigma_i^2(t):=\Var (Y_i(t))$
are regularly varying at $\infty$ with indexes $\alpha_i\in(0,2)$.
Without loss of generality we assume that
$0<\alpha_1\le\ldots\le\alpha_n<2$.
Moreover, assume that there exists $\kappa\in\{1,\ldots,n\}$ such that
$\sigma_1^2\sim\ldots\sim c_\kappa\sigma_\kappa^2$ for some $c_i>0$ and
$\lim_{t\toi}\sigma_{\kappa}(t)/\sigma_{\kappa+1}(t)=0$
(if $\kappa=1$, then set $c_\kappa=1$; if the first condition is satisfied with $\kappa=n$,
then the second one is redundant).

\vspace{2mm}

{\bf C3} \ \ \ $\lim_{t\to 0} \sigma_i^2(t)|\log|t||^{1+\varepsilon}<\infty$ for
some $\varepsilon>0$.

\vspace{2mm}

We analyze
\begin{equation}
\label{multi.reg}
\pp\left(\exists{t\ge 0}:\bs X(t)-\bs i(t)\ge u\bs q\right).
\end{equation}
Probabilities of this type play an important role in risk theory, describing the probability of simultaneous ruin of multiple (dependent) companies;
see \citet{Avram08} for related results.
The one-dimensional counterpart of (\ref{multi.reg})
was considered in \citet{Debicki99} in the context of Gaussian fluid models.
Related examples and further references can be found in the monograph \cite{Mandjes07}.
In the following proposition we derive the logarithmic asymptotics of \eqref{multi.reg}.

With $c_i$ as in {\bf C2}, set
\[
C:=\diag(1,c_2,\ldots,c_\kappa,0,\ldots,0)
\]
and
\[
J(C,S,\bs q,\alpha):=
\inf_{t\ge0}\inf_{ \bs v\ge\bs  q} \frac{\scal{ S^{-1}(\bs v+ \bs i (t))}{CS^{-1} (\bs v+ \bs i(t))}}{t^{\alpha}}.
\]
\begin{proposition}
\label{prop:reg}
Assume that
$\bs Y$ satisfies {\bf C1}-{\bf C3}, and $S$ is
an invertible matrix.
Then, for $\{\bs X(t):t\ge 0\}:=\{S\bs Y(t):t\ge 0\}$,
\[
\log\pp\left(\exists{t\ge 0}:\bs X(t)-\bs i(t)\ge u\bs q\right)\sim -
\frac{u^2}{2\sigma_1^2(u)}J(C,S,\bs q,\alpha_1),\as u.
\]
\end{proposition}
\begin{proof}
We start by checking that {\bf A1}-{\bf A2} are satisfied for
$(\bs X, {\bs i})$.
Indeed, let us note that
the matrix $K_t=K$ is constant.
Besides, since $S$ is invertible, then $K$ is invertible too, which
combined with the fact that $K$ is positive-definite and
$k_{i,i}=1$, straightforwardly implies that assumption {\bf A1} is satisfied.

Since $\bs Y$ has stationary increments, then
under {\bf C1}-{\bf C3} $\lim_{t\toi}Y_i(t)/t= 0$ almost surely and therefore  (using that $\bs X$ consists of linear combinations
of the $Y_i$, $i=1,\ldots,n$) assumption {\bf A2} is complied with; see \cite[Lemma 3]{Dieker05b}  for details.
Now following Theorem \ref{thm:main},
\begin{align*}
M_{\bs X,\bs i,\bs q}(u;[0,\infty))&=\half\inf_{t\ge0}\inf_{\bs v\ge u \bs q} \scal{ S^{-1}(\bs v+\bs i(t))}{R_t^{-1} S^{-1}(\bs v+\bs i(t))}\\
&=
\half\inf_{t\ge0}\inf_{\bs v\ge\bs q} \scal{ S^{-1}(u\bs v+u\bs i(t))}{R_{ut}^{-1} S^{-1}(u\bs v+u\bs i(t))}\\
&=
\frac{u^2}{2}\inf_{t\ge0}\inf_{\bs v\ge\bs q} \scal{ S^{-1}(\bs v+\bs i(t))}{R_{ut}^{-1}S^{-1} (\bs v+\bs i(t))},
\end{align*}
where the matrix $R_{t}^{-1}$ equals $\diag(\sigma_1^{-2}(t),\ldots,\sigma_n^{-2}(t))$, which
is the inverse of the covariance matrix of $\bs Y$.
Using the regular variation of $\sigma_i^2(\cdot)$, we find that, as $u\toi$,
\[
\sigma_1^2(u)R_{ut}^{-1}\to t^{-\alpha_1}C,\as u.
\]
By virtue of the uniform convergence theorem we arrive at
\[
M_{\bs X,\bs i,\bs q}(u;[0,\infty))
\sim
\frac{u^2}{2\sigma_1^2(u)}
\inf_{t\ge0}\inf_{\bs v\ge\bs q} \frac{\scal{ S^{-1}(\bs v+\bs i(t))}{CS^{-1} (\bs v+\bs i(t))}}{t^{\alpha_1}},
\]
as $u\toi$. This completes the proof.
\end{proof}

\section{The proof of the main theorem}
\label{PMT}
This section is devoted to the proof of our main result -- Theorem \ref{thm:main}. 
We will achieve this by establishing an upper bound and
a lower bound.
We start by presenting the following `saddle point equality' that plays a
crucial role in the upper bound.
\begin{lem}
\label{lemma:saddle}
Let $A$ be any positive-definite matrix. Then,
\[
\sup_{\bs w\in\rr^n_+}\frac{\scal { \bs w}{\bs q} ^2}{\scal {\bs w} {A\bs w}}
=
\inf_{\bs v\ge\bs q} \scal {\bs v} {A^{-1} {\bs v}},
\]
for any vector $\bs q\in\rr^n_+$. Moreover, if $\bs v{^\star}$ is the optimizer of the infimum problem in the right-hand side, then
$\bs w{^\star}:=A^{-1}\bs v{^\star}$ is an optimizer of the supremum problem in the left-hand side.
\end{lem}
\begin{proof}
Decompose $A=BB'$ for some nondegenerate matrix $B$. Then,
\[
\frac{\scal{\bs w} {\bs q} ^2}{\scal {\bs w} {A\bs w}}=\frac{\scal{\bs w}{\bs q} ^2}{\|B'\bs w\|^2}
\quad\text{and}\quad
\scal{\bs v} {A^{-1}\bs v} = \|B^{-1}\bs v\|^2.
\]
Now, for $\bs w\in\rr^n_+$, the Cauchy-Schwarz inequality yields
\[
\scal{\bs w}{\bs q} =\inf_{\bs v\ge\bs q} \scal{\bs w}{\bs v} = \inf_{\bs v\ge\bs q} \scal { B'\bs w} {B^{-1}\bs v}
\le \|B'\bs w\|\inf_{\bs v\ge\bs q} \|B^{-1}\bs v\|.
\]
Dividing both sides by $\|B'\bs w\|>0$ and optimizing the left-hand side of the previous display, we arrive at
\[
\sup_{\bs w\in\rr^n_+}\frac{\scal {\bs w}{\bs q} ^2}{\scal {\bs w} {A\bs w}}
\le \inf_{\bs v\ge\bs  q} \scal {\bs v} {A^{-1} \bs v}.
\]
To show the opposite inequality, assume that $\bs v{^\star}$ is such that
\[
\inf_{\bs v\ge\bs  q} \scal {\bs v} {A^{-1}\bs v}=
\scal {\bs v{^\star}} {A^{-1} \bs v{^\star}}.
\]
The Lagrangian function of the above problem is given by $L(\bs v,\bs \lambda):=
 \scal {\bs v} {A^{-1}\bs v}-\scal{\bs \lambda} {\bs v-\bs q}$ for $\bs \lambda\ge\bs0$, and
due to complementary-slackness considerations
we necessarily have that $A^{-1}\bs v{^\star}\ge\bs 0$,
and if $(A^{-1}\bs v{^\star})_i>0$, then $v{^\star}_i=q_i$. Thus take $\bs w{^\star}=A^{-1}\bs v{^\star}\in\rr^n_+$,
so that
\[
\frac{\scal {\bs w{^\star}}{\bs q} ^2}{\scal {\bs w{^\star}} {A\bs  w{^\star}}}=
\frac{\scal { A^{-1}\bs v{^\star}}{\bs q} ^2}{\scal { A^{-1}\bs v{^\star}} {\bs v{^\star}}}= \scal {\bs v{^\star}} {A^{-1} \bs v{^\star}}.
\]
Indeed, the last equality is equivalent to
\[
\scal { A^{-1}\bs v{^\star}}{ \bs q-\bs v{^\star}}= 0,
\]
but recall that if $(A^{-1}\bs v{^\star})_i\ne 0$, then $(\bs q-\bs v{^\star})_i=0$. Hence finally,
\[
\sup_{\bs w\in\rr^n_+}\frac{\scal {\bs w}{\bs q} ^2}{\scal {\bs w} {A\bs  w}}\ge\inf_{\bs v\ge\bs  q} \scal {\bs v} {A^{-1}\bs  v},
\]
which proves the opposite inequality. This finishes the proof.
\end{proof}

The main idea behind the proof of the upper bound of Theorem \ref{thm:main}
is that the $\rr^n$-valued process $\bs X(\bt)-\bs d(\bt)$ can be effectively replaced
by a suitably chosen {\it $\rr$-valued} Gaussian process.
The asymptotics of the latter process can then be handled using the familiar techniques for real-valued Gaussian processes.

For any vector $\bs w\in \rr^n_+$, define
\[
Z_{u,\bs w}(\bt):=\frac{\scal{\bs w}{\bs X(\bt)}}{\scal{\bs w}{u\bs q+\bs d(\bt)}},\]
and observe that (with $u>u_0$; cf. Remark \ref{r4})
\[
\pp\left(\exists{\bt\in T}:\bs X(\bt)- \bs d(\bt)>u\bs q \right)
\le
\pp\left(\sup_{\bt\in T} Z_{u,\bs w}(\bt)>1\right).
\]
The vector $\bs w$ in the process $Z_{u,\bs w}$ can be seen as a vector of weights
assigned to the coordinates of $\bs X$. For fixed $u$ and $\bs w$ the process $Z_{u,\bs w}$
is a centered Gaussian process. We shall show that it also has almost surely bounded sample paths.
\begin{lem}
\label{lem:Zuw}
Under {\bf A1}-{\bf A2}, the process $Z_{u,\bs w}$
is a centered Gaussian process with bounded sample paths almost surely,
for each $\bs w\in \rr^n_+$ and $u>u_0$. Moreover,
\[
\sup_{\bt\in T}Z_{u,\bs w}(\bt)\pto0\as u.
\]
\end{lem}
\begin{proof}
Without loss of generality we can assume that $\|\bs w\|=1$.
For any $L\ge1$, recalling the definition of $\bs\ell$ from Remark \ref{r4},
\begin{eqnarray*}
\lefteqn{\hspace{-1cm}
\pp\left(\sup_{\bt\in T} Z_{u,\bs w}(\bt)>L\right)=
\pp\left(\exists{\bt\in T}: \scal{\bs w}{\bs X(\bt)}>\scal{\bs w}{Lu\bs q+L\bs \ell+L(\bs d(\bt)-\bs \ell)}\right)}\\
&\le&
\pp\left(\exists{\bt\in T}: \scal{\bs w}{\bs X(\bt)}>\scal{\bs w}{L(u\bs q+\bs \ell)+(\bs d(\bt)-\bs \ell)}\right)\\
&\le&
\pp\left(\exists{\bt\in T}: \scal{\bs w}{\bs X(\bt)-\bs d(\bt)}>\scal{\bs w}{L(u\bs q+\bs \ell) -\bs \ell}\right)\\
&\le&
\pp\left(\sum_{i=1}^nw_i\sup_{\bt\in T}(X_i(\bt)-d_i(\bt))>
L\scal{\bs w}{u\bs q+\bs \ell}-\scal {\bs w}{\bs \ell}\right)\\
&\le&
\pp\left(\sum_{i=1}^n\sup_{\bt\in T}(X_i(\bt)-d_i(\bt))^+>
L\min_i(uq_i+\ell_i)/\sqrt n -\|\bs \ell\|\right),
\end{eqnarray*}
where the last probability tends to zero with $L\toi$ due to {\bf A2}.
This proves that $Z_{u,\bs w}$ has bounded sample paths almost surely.

The last probability also tends to zero with $L\ge 1$ fixed and $u\toi$. On the other hand,
for any $L<1$ we have
\begin{align*}
\pp\left(\sup_{\bt\in T} Z_{u,\bs w}(\bt)>L\right)&=
\pp\left(\exists{\bt\in T}: \scal{\bs w}{\bs X(\bt)-L\bs d(\bt)}>L\scal{\bs w}{u\bs q}\right)\\
&\le
\pp\left(\sum_{i=1}^n\sup_{\bt\in T}(X_i(\bt)-Ld_i(\bt))^+>uL\min_i q_i/\sqrt n\right),
\end{align*}
where the last probability tends to zero with $u\toi$ by virtue of {\bf A2}.
We therefore have that $\sup_{\bt\in T} Z_{u,\bs w}(\bt)$ converges to $0$ in probability.
\end{proof}

The above considerations remain true even if $\bs w$ depends on $u$ and $\bt$. 
This observation allows us to optimize the variance of the
process $Z_{u,\bs w}$, while retaining its sample path properties. Notice that
\[
\Var (Z_{u,\bs w}(\bt))=\frac{\scal{\bs w} {\Sigma_{\bt}\bs w}}{\scal{\bs w} {u\bs q +\bs d(\bt)}^2}.
\]
Therefore, take $\bs w{^\star}\equiv \bs w{^\star}(u,\bt)$ such that
\begin{equation}
\label{eq:w}
\frac{\scal {\bs w{^\star}} {\Sigma_\bt \bs w{^\star}}}{\scal {\bs w{^\star}} {u\bs q +\bs d(t)}^2}=
\inf_{\bs w\in \rr^n_+}\frac{\scal {\bs w} {\Sigma_{\bt}\bs w}}{\scal {\bs w} {u\bs q +\bs d(\bt)}^2}
\end{equation}
and denote by
$Y_u(\bt)$ the process $Z_{u,\bs w{^\star}}(\bt)$ with the weights $\bs w=\bs w{^\star}$ chosen as above.
Let $\sigma_u^2(\bt)$ be the variance function of the process $Y_u(\bt)$. Then,
by Lemma \ref{lemma:saddle},
\begin{equation}
\label{eq:Yvar}
\sigma_u^{-2}(\bt)=M_{\bs X,\bs d,\bs q}(u,\bt).
\end{equation}

To estimate the tail of the supremum of the process $Y_u(\bt)$ we intend to use Borell's inequality \cite[Theorem 2.1]{Adler90}.
To apply this result, we need to verify that the expectation of $\sup_{\bt\in T} Y_u(\bt)$ vanishes as $u\toi$.
This is done in the next lemma.
\begin{lem}
\label{lem:impl}
Under {\bf A1}-{\bf A2}, with $u_0$ as in Remark \ref{r4},
\begin{enumerate}
\item $M_{\bs X,\bs d,\bs q}(u;T)>0$ for each $u>u_0$;
\item $\lim_{u\toi} M_{\bs X,\bs d,\bs q}(u;T)=\infty$;
\item $\lim_{u\toi} \ee \sup_{\bt\in T}Y_u(\bt)=0$.
\end{enumerate}
\end{lem}
\begin{proof}
From Lemma \ref{lem:Zuw} we know that for a fixed $u$ the process
$Y_u$ has bounded sample paths almost surely. This implies that
$\sup_{\bt\in T} \sigma_u^2(\bt)<\infty$. But
\[
\sup_{\bt\in T}\sigma_u^2(\bt)=\sup_{\bt\in T}(M_{\bs X,\bs d,\bs q}(u,\bt))^{-1}=\half(M_{\bs X,\bs d,\bs q}(u;T))^{-1}
\]
and claim (1) follows.

The proof of (2) is a consequence of the fact that under {\bf A2}
\[
\pp\left(\sup_{\bt\in T} Y_u(\bt)>1\right)\to0\as u,
\]
and for $\cn$ being a standard normal random variable
\[
\pp\left(\sup_{\bt\in T} Y_u(\bt)>1\right)\ge
\sup_{\bt\in T}\pp\left(Y_u(\bt)>1\right)=
\pp\left(\cn>\inf_{\bt\in T} \sqrt{M_{\bs X,\bs d,\bs q}(u,\bt)}\right).
\]

To prove the last claim, observe that the
almost sure boundedness of sample paths of $Y_u(\bt)$ implies
that $\ee\sup_{\bt\in T} Y_u(\bt)<\infty$ and it
easily follows that the family $(\sup_{\bt\in T}Y_u(\bt))_u$ is uniformly integrable.
Now claim (3) follows from the second part of Lemma \ref{lem:Zuw}.
\end{proof}

Before we proceed to the proof of Theorem \ref{thm:main} we state a technical lemma,
which is a prerequisite for the proof of the lower bound.
\begin{lem}
\label{lemma:lower}
Under {\bf A1}, there exist constants $C_1<\infty$, $C_2>0$ such that for any $\bt\in T$
\[
\log\pp\left(\bs X(\bt)-\bs d(\bt)>u\bs q\right)\ge -\half M_{\bs X,\bs d,\bs q}(u,\bt)-C_1 M_{\bs X,\bs d,\bs q}^{1/2}(u,\bt)+C_2.
\]
\end{lem}

\begin{proof}
Set \[Q_{\bt}:=\{\bs x\in\rr^n:\bs x > u\bs q+\bs d(\bt)\},\] and
let $B_{\bt}$ be  such that $B_{\bt}B_{\bt}'=\Sigma_{\bt}$. Then $X(\bt)\stackrel{\rm d}{=} B_{\bt}\cn$,
where $\cn$ is an $\rr^n$-valued standard normal random variable with the density function
\[
f(\bs x)=D_n\exp\left(-\half\scal{\bs x} {\bs x}\right),
\]
for some normalizing constant $D_n$. In this notation, we have
\[
\pp\left(\bs X(\bt)-\bs d(\bt)>u\bs q\right)=\pp\left(\bs X(\bt)\in Q_{\bt}\right)
=
\pp\left(\cn\in B_{\bt}^{-1}Q_{\bt}\right).
\]
Now let $\bs x{^\star}=\bs x{^\star}(u,\bt)\in B_{\bt}^{-1}Q_{\bt}$  be such that
\[
M_{\bs X,\bs d,\bs q}(u,\bt)=\inf_{\bs x\in Q_\bt}\scal{\bs x} {\Sigma^{-1}_{\bs t}\bs x}
=\inf_{\bs x\in B_{\bs t}^{-1}Q_\bt}\scal{\bs x}{\bs x}
=\scal {\bs x{^\star}}{\bs x{^\star}},
\]
and let $A_{\bt}:=\mathcal{B}(x{^\star},1)\cap B_{\bt}^{-1}Q_{\bt}$, where
$\mathcal{B}(x{^\star},1)$ is a ball in $\rr^n$ of radius $1$ and center $x{^\star}$.
Then,
\[
\pp\left(\cn\in B_{\bt}^{-1}Q_{\bt}\right)
\ge
\int_{A_{\bt}}f(\bs x)\,d\bs x.
\]
Set $\Delta(\bs x,\bs x{^\star}):=\scal {\bs x} {\bs x}-\scal{\bs x{^\star}} {\bs x{^\star}}.$
Then
\[
\pp\left(\cn\in B_{\bt}^{-1}Q_{\bt}\right)
\ge D_n\Vol(A_{\bt})
\exp\left(-\half M_{\bs X,\bs d,\bs q}(u,\bt)-\half\sup_{\bs x \in A_\bt}\Delta(\bs x, \bs x{^\star})\right).
\]
Since
\[
\Delta(\bs x,\bs x{^\star})
\le
2\|\bs x-\bs x{^\star}\|\scal{\bs x{^\star}} {\bs x{^\star}}^{1/2}+\|\bs x-\bs x{^\star}\|^2,
\]
we have that
\[
\sup_{\bs x\in A_\bt} \Delta(\bs x,\bs x{^\star})
\le
2\diam(A_\bt)M_{\bs X,\bs d,\bs q}^{1/2}(u,\bt)+\diam^2(A_\bt).
\]
Therefore the claim follows if $\diam(A_\bt)$ and $\Vol(A_\bt)$ can be bounded
uniformly in $\bt\in T$ from above and below, respectively.

Observe that, by the construction of $A_\bt$,
$\diam(A_\bt)\le 1$.
Besides, the quadrant $Q_\bt$ is spanned by the standard basis $(\bs e_i)$ in $\rr^n$
fixed in the point $u\bs q+\bs d(\bt)$. The cosine of the angle $\alpha_{i,j}$
between $B^{-1}_\bt\bs e_i$ and $B^{-1}_\bt\bs e_j$ is given by
$\cos(\alpha_{i,j})=k_{i,j}$;
see Remark \ref{rmk:mainassum}. Under $\bf{A1}$ this angle
is bounded away from zero, uniformly in $\bt\in T$. Therefore
$\inf_{\bt\in T}\Vol(A_\bt)>0$.
This completes the proof.
\end{proof}

Now we are ready to prove the main theorem.
\begin{proof}[Proof of Theorem \ref{thm:main}]
Put $P(u):=\pp\left(\exists{\bt\in T}:\bs X(\bt)-\bs d(\bt)>u\bs q\right)$.
We split the proof into two parts: the lower and the upper bound.\\
Lower bound: The lower bound follows directly from Lemma \ref{lemma:lower}
and the inequality
\[
\log P(u)\ge\sup_{\bt\in T}\log\pp\left(\bs X(\bt)-\bs d(\bt)>u\bs q\right).
\]
Upper bound:
Let $\bs w{^\star}:\rr_+\times T\to\rr^n_+$ be the mapping chosen in \eqref{eq:w}.
Now as in the definition of the process $Y_u$,
\begin{align*}
P(u)
&\le
\pp\left(\exists{\bt\in T}: \scal{\bs w{^\star}}{\bs X(\bt)} >\scal{\bs w{^\star}}{u\bs q+\bs d(\bt)}\right)\\
&=
\pp\left(\sup_{\bt\in T}\frac{\scal{\bs w{^\star}}{\bs X(\bt)}}{\scal{\bs w{^\star}}{u\bs q+\bs d(\bt)}}> 1\right)
=
\pp\left(\sup_{\bt\in T}Y_u(\bt)> 1\right),
\end{align*}
where the passage from the $n$-dimensional quadrant to the tangent increases the probability.
Recall that the variance  $\sigma^2_u(\bt)$ of $Y_u(\bt)$ equals $(M_{\bs X,\bs d,\bs q}(u,\bt))^{-1}$; cf. \eqref{eq:Yvar}.
Moreover, thanks to Lemma \ref{lem:impl}, the Gaussian process $Y_u$ has bounded sample paths almost surely.
Therefore, Borell's inequality implies that
\[
\pp\left(\sup_{\bt\in T}Y_u(\bt)> 1\right)
\le
2\exp\left(-\left(1-\ee\sup_{\bt\in T}Y_u(\bt)\right)^2M_{\bs X,\bs d,\bs q}(u;T)\right).
\]
Now from (2) and (3) of Lemma \ref{lem:impl} we obtain
\[
\limsup_{u\toi}\frac{\log\pp\left(\sup_{\bt\in T}Y_u(\bt)> 1\right)}{M_{\bs X,\bs d,\bs q}(u;T)}\le -1
\]
and the claim follows.
\end{proof}

\begin{rmk}
From the proof of the upper bound we obtain the  useful inequality
\[
\pp\left(\exists{\bt\in T}:\bs w{^\star}X(\bt)>\bs w{^\star}(u\bs q+\bs d(\bt))\right)
\le
\pp\left(\exists{\bt\in T}: \scal{\bs w{^\star}}{\bs X(\bt)} >\scal{\bs w{^\star}}{u\bs q+\bs d(\bt)}\right),
\]
which we have proven to be exact in terms of logarithmic asymptotics.
Let $\bs v{^\star}\equiv\bs v{^\star}(u,\bt)$ be such that
\[
\scal {\bs v{^\star}+\bs d(\bt)} {\Sigma_{\bt}^{-1} (\bs v{^\star}+\bs d(\bt))}=
\inf_{\bs v\ge u\bs q}\scal {\bs v+\bs d(\bt)} {\Sigma_{\bt}^{-1} (\bs v+\bs d(\bt))}.
\]
Then the optimal weights $\bs w{^\star}$ are given by $\bs w{^\star}(u,\bt)=\Sigma_{\bt}^{-1}\bs v{^\star}(u,\bt),$
or alternatively, due to Lemma \ref{lemma:saddle}, by
\[
\bs w{^\star}(u,\bt)=\arg\sup_{\bs w\in \rr^n_+}\frac{\scal{\bs w}{u\bs q+\bs d(\bt)}^2}{\scal{\bs w}{\Sigma_{\bt}\bs w}}.
\]
Observe that the weights do not depend on $u$ in the case of $\bs d\equiv \bs 0$.
\end{rmk}

\newpage
\bibliography{Gaussian}
\end{document}